\documentclass[reqno]{alea2}
\usepackage{natbib}
\usepackage{fancyhdr}
\usepackage{graphicx}

\usepackage{amssymb}
\usepackage{amsmath}

\renewcommand{\cite}{\citet}

\makeatletter \@addtoreset{equation}{section} \makeatother

\renewcommand\theequation{\thesection.\arabic{equation}}
\renewcommand\thefigure{\thesection.\@arabic\c@figure}
\renewcommand\thetable{\thesection.\@arabic\c@table}

\theoremstyle{plain}
\newtheorem{theorem}{Theorem}[section]                                                                    
\newtheorem{lemma}[theorem]{Lemma}

\theoremstyle{definition}

\theoremstyle{remark}

\newcommand{\1}{{\rm 1}\kern-0.24em{\rm I}} 

\begin{document}

\date{March 31th 2009, accepted }

\author{Julien Sohier }
\address{Universit\'{e} Paris Diderot, LPMA,  Math\'{e}matiques, case 7012; 2,place Jussieu, 75251 Paris Cedex 05, France}
\email{sohier@math.jussieu.fr}

\title[Finite size scaling for homogeneous pinning models]{Finite size scaling for homogeneous pinning models} 

\begin{abstract}
  
Pinning models are built from discrete renewal sequences
by rewarding (or penalizing) the trajectories according
to their number of renewal epochs up to time $N$, and $N$ is then
sent to infinity. They are
statistical mechanics models to which a lot of attention has been paid
both because they are very relevant for applications
and because of their {\sl exactly solvable character}, while displaying
a non-trivial phase transition (in fact, a localization transition).
The order of the transition depends on the tail of the inter-arrival
law of the underlying renewal and  the transition is continuous
when such a tail is sufficiently heavy: this is the case on which we 
will focus. The main purpose of this work is
to give a mathematical treatment of the {\sl finite size scaling limit} 
of pinning models, namely studying
the limit (in law) of the process close to criticality when the system 
size is proportional to the correlation length. \\
 
\end{abstract}

\maketitle

\section{Introduction}

   \indent The model of directed polymers interacting with a one dimensional defect turns out to be quite satisfactory to describe various biological or physical phenomena, such as $(1+1)$ interface wetting \cite{Der}, the problem of depinning of flux lines from columnar defects in type II superconductors \cite{Ne} and the denaturation transition of DNA in the Poland Sheraga approximation \cite{Ka}. The first results obtained by the physical community have been summarized in \cite{Fis}, a paper which received a lot of attention from mathematicians. Since then, this model has been under close  scrutinity, including in the last few years (the recent mathematical monograph \cite{GG}, but also \cite{CGZ}, \cite{IY} or \cite{DGZ}). Here, we are going to deal with the homogeneous version of the model; this version is mathematically remarkable in the sense that it is essentially completely solvable. For example, the model can have a transition of any given order \cite{GG} depending on the value of $\alpha$, a parameter in the definition of the system ($1 + \alpha$ is sometimes called loop exponent in the literature).\\ 
    \indent A central role in statistical mechanics is played by the correlation length, a quantity that we will denote by $\xi$. Properly defining this quantity in an (infinite volume) model requires some work and in most of the cases, such a concept does not correspond to a unique mathematical object. All the same, it is essential that any ``reasonably defined'' correlation length behaves in the same way close to criticality. And precisely close to criticality, for a large class of models, $\xi$ becomes the only ``relevant'' scale in the system (see below for a more precise explanation). This idea is one of the basic concepts of the so called ``finite size scaling'' theory (see \cite{Ca}). Here we  illustrate in very concrete terms this concept via the scaling limit of pinning systems, obtaining a limiting behavior if the parameter is in a small, size dependent window near criticality as the size of the system goes to $\infty$. \\ 
 \indent In this first part, we will try to see in which way the above, rather unprecise statements, can be made  quantitative. The second part should be seen as a warm up for the two last parts; actually, in the one dimensionnal wetting, the existing literature makes the statement quite easy. In the third part, we introduce our general return model, and make the precise statements of our result in terms of scaling limits of the zero level set of our system near criticality. \\ \indent In this note, $\alpha$ will always be a real number in $(0,1)$. For positive sequences $(a_n)_{n \geq 0}$ and $(b_n)_{n \geq 0}$, we will write $a_n \sim b_n$ if $\lim_{n \to \infty} a_n/b_n = 1 $, and for  random variables $(X_n)_{n \geq 0}$ and $X$ with values in a Polish space $(E,d)$, we will write $X_n \Rightarrow X$ if the sequence $(X_n)_{n \geq 0}$ converges in law towards $X$. \\
 
    \indent 1.1 \textbf{A first model.}  
     Let $\tau$ be a recurrent renewal process with law $\mathbf{P}$, that is $\tau = \{ \tau_n \}_{n \geq0}$ where $\tau_0 = 0$ and $(\tau_{i+1} - \tau_i)_{i \geq0}$ are iid with common law $\mathbf{P}$, where $\mathbf{P}$ is $\mathbf{N}$ valued and verifying, as $k \rightarrow \infty$, 

     \begin{equation} \label{P}
 \mathbf{P} (\tau_{i+1} - \tau_i = k)  \sim \frac{C}{k^{1+\alpha}}.
     \end{equation}
     
  \indent
 We will actually consider a slightly more general model in section 3, allowing in particular $\tau$ to be transient, but for simplicity, we will restrict ourselves temporarly to this setup.
  The set $\tau \cap [0,N]$ can be considered as a random subset of  $\{0,1,\ldots,N \}$. We define the law $\mathbf{P}_{N,\beta}$ on  the subsets of $\{0,1,\ldots,N \}$ by 
   \begin{equation} \label{Def} 
    \frac{d\mathbf{P}_{N,\beta}}{d\mathbf{P}} (\tau) = \frac{1}{Z_{N,\beta}} \exp( \beta \mathcal{N}_N(\tau)),
    \end{equation}  
     where  $\mathcal{N}_N(\tau)$ is the cardinality of the set $\tau \cap [0,N]$. In \eqref{Def}, $ Z_{N,\beta}$ is the partition function of the model and of course  
   \begin{equation}
 Z_{N,\beta} = \mathbf{E} \left[ \exp(\beta \mathcal{N}_N(\tau)) \right].
\end{equation}
   \indent It is not difficult to show that the limit of the quantity $F_N(\beta) := \frac{1}{N} \log{ Z_{N,\beta}}$ exists and is non negative. We denote it by $F(\beta)$. $F(\cdot)$ is a convex, non decreasing and non negative function, and there exists a critical value of $\beta$ such that $ \beta \leq \beta_c$ implies $F(\beta) = 0$ and $\beta > \beta_c$ implies the positivity of $F(\beta)$. If $F(\beta) = 0$, the system is said to be \textit{delocalized}, and it is \textit{localized} otherwise. As the underlying renewal is recurrent, we actually have $\beta_c = 0$ (see \cite{GG}, chapter 2 and section 3 below).\\

 \indent 1.2 \textbf{Finite size scaling.} We define the correlation length of our system  by 
 \begin{equation}
 \xi(\beta) := \frac{1}{F(\beta)},
\end{equation}  
 if $\beta > 0$, and $\xi(\beta) = \infty$ otherwise. This definition has first been introduced in the physical litterature in \cite{Fis}, and its mathematical relevance with respect to other quantities (in particular with the {\sl correlation function}) is discussed in \cite{GG1}. Also note that the concept of correlation length has been considered in depth in the inhomogeneous case (that is when the reward-penalty become site-dependent and random), see e.g. \cite{To}. It has been shown in \cite{GG} that  
  \begin{equation} \label{F}
 \xi(\beta) \stackrel{ \beta \searrow \beta_c}{\sim }c_{\alpha} (\beta - \beta_c)^{-1/\alpha},
\end{equation} 
 for some (explicit) constant $ c_{\alpha} >0$. \\
 \indent 
  We expect finite size effects to appear only if the system is of the scale of the correlation length, that is only if $\xi(\beta)/N$ stays fixed (say equal to $q$). It is clear (considering \eqref{F}) that this is equivalent to keeping $(\beta - \beta_c)^{1/\alpha}N$ fixed (at least close to criticality). In this case, the observables of our system have a different behavior near criticality than their behavior in each of the two phases in the infinite volume limit. In an alternative way, we could say that close to criticality, the only relevant length scale is the correlation length. In order to make this precise, we need two basic objects. The first one is the \textit{stable subordinator} with index $\alpha$. Recall that a subordinator is an increasing L\'{e}vy process; a subordinator $(\sigma^{(\alpha)}_{s})_{s \geq 0}$ is said to be \textit{$\alpha$ stable} (for $\alpha \in (0,1)$) if, for all $M \in \mathbf{R}^+$, the process  $(\sigma^{(\alpha)}_{Ms})_{s \geq 0}$ has the same law than the process $(M^{1/\alpha}\sigma_s^{(\alpha)})_{s \geq 0}$. The second one is its local time $(L^{(\alpha)}_t)_{t \geq 0}$, which is simply the generalized inverse of $(\sigma^{(\alpha)}_{s})_{s \geq 0}$, that is 
\begin{equation} \label{DefTpsLoc}
 L^{(\alpha)}_t := \inf \left\{ s \geq 0, \sigma^{(\alpha)}_{s} \geq t \right\}.
\end{equation} 
  It will be shown in the last part (Lemma 1) that there exists $C_K > 0$ such that 
 \begin{equation}
 \left(C_K \frac{ \mathcal{N}_N}{N^{\alpha}} \right)_{N \geq 0} \Rightarrow L^{(\alpha)}_1,
\end{equation}
 and moreover the sequence $\left(C_K \frac{ \mathcal{N}_N}{N^{\alpha}} \right)_{N \geq 0}$ is uniformly integrable (Lemma 2). Using these two facts and \eqref{F}, one gets that the \textit{finite size correlation length} $\xi_N(\beta) := 1/F_N(\beta)$ verifies 
 \begin{equation} \label{eq1}
  \frac{\xi_N(\beta)}{N} \stackrel{ N \to \infty } \sim  \left( \log \mathbf{E} \left[ \exp \left( q^{\alpha}
c_{\alpha}^{\alpha} \frac{ \mathcal{N}_N}{N^{\alpha}} \right) \right] \right)^{-1} \sim f(q),    
\end{equation} where $f(q) = \left( \log \mathbf{E} \left[ \exp \left( \frac{q^{\alpha}
c_{\alpha}^{\alpha}}{C_K} L^{(\alpha)}_1 \right) \right] \right)^{-1}$. \\ \indent 
   Let us now consider two very relevant observables (in the usual window): the (finite volume) \textit{contact fraction} $\rho_N(\beta)$ (that is the expectation of $\mathcal{N}_N/N$ with respect to the finite polymer measure) and the (once again finite volume) \textit{specific heat} $\chi_N(\cdot)$ (that is the second derivative of the finite volume free energy). In our setup, these quantities can be written in an analogous way as above. Specifically, we get
  \begin{equation}
  \rho_N(\beta) \sim N^{\alpha -1} \frac{1}{C_K}  \frac{ \mathbf{E} \left[ L^{(\alpha)}_1 \exp \left( \frac{q^{\alpha}
c_{\alpha}^{\alpha}}{C_K} L^{(\alpha)}_1 \right) \right] }{ \mathbf{E} \left[ \exp \left( \frac{q^{\alpha}
c_{\alpha}^{\alpha}}{C_K} L^{(\alpha)}_1 \right) \right]} := N^{\alpha -1} \hat{\rho}(q).
\end{equation}
  \indent In an equivalent way, introducing $\tilde{\rho}(x) := x^{1 - \alpha} \hat{\rho}(q)$, we get  $\rho_N(\beta) \sim \xi(\beta)^{\alpha -1} \tilde{\rho}(q)$, that is we can express  $\rho_N(\cdot)$ (as $ N \to \infty$) as a function of $\xi(\cdot)$. Analogously, we have 
 \begin{equation}
  \frac{\partial^2}{\partial^2 \beta} F_N(\beta) = \chi_N(\beta) \sim N ^{2 \alpha -1} \frac{1}{C_K^2}\frac{ \mathbf{E} \left[ (L^{(\alpha)}_1 - \mathbf{E} [ L^{(\alpha)}_1 ])^2 \exp \left( \frac{q^{\alpha}
c_{\alpha}^{\alpha}}{C_K} L^{(\alpha)}_1 \right) \right] }{ \mathbf{E} \left[ \exp \left( \frac{q^{\alpha}
c_{\alpha}^{\alpha}}{C_K} L^{(\alpha)}_1 \right) \right]},
\end{equation}
 so that, once again, we get the following equivalence: 
 \begin{equation}
 \chi_N(\beta) \sim \xi(\beta)^{2 \alpha -1} \hat{\chi}(q).
\end{equation}
 \indent We show here that these two particular (and physically relevant) examples are actually special cases for a much more general phenomenon, namely the convergence in law of the whole rescaled system as $\xi( \beta)/N$ stays fixed. We are able to compute the scaling limit of the system in terms of a Radon Nykodym derivative with respect to the $\alpha$ regenerative set. \\

 \section{ A look at the wetting model}

   As a warm-up, we will deal with the case in which the underlying law of our polymer model is the simple symmetric random walk conditioned to stay non negative, that is a sequence $(S_n)_{n \geq 0}$ where $S_0 = 0$, the variables $(S_i - S_{i-1})_{i \geq 1}$ are iid and such that $\mathbf{P} \left[ S_i - S_{i-1} =1 \right] = \mathbf{P} \left[ S_i - S_{i-1} = -1 \right] = \frac{1}{2}$.
  We introduce the following probability law on $ \mathbf{Z}^N$:    
\begin{equation}
  \frac{ d \mathbf{P}_{\beta,N}}{ d\mathbf{P}} (x)  =
  \frac{1}{Z_{\beta,N}} \exp(H_{\beta,N}(x)) \1_{ x_1 \geq 0, \ldots, x_N \geq 0 } 
\end{equation} 
 where the Hamiltonian of our system is given by
 
\begin{equation}
  H_{\beta,N}(x) = \beta\sum_{i=1}^N
  \1_{x_i=0} = \beta \mathcal{N}_N.
\end{equation}
 \indent 
  Here also, it is not hard to prove the existence of $\lim_{N \rightarrow \infty} \frac{1}{N} \log( Z_{\beta,N}) = F(\beta)$ (see \cite{GG}). In particular, for $N$ even, we have the inequality 
 \begin{equation*}
  Z_{\beta,N} = \mathbf{E} \left[ \exp(H_{\beta,N}(S)) \1_{ S_1 \geq 0, \ldots, S_N \geq 0 } \right]
\end{equation*}
  \begin{equation}
  \geq \mathbf{E} \left[ \exp(H_{\beta,N}(S)) \1_{ S_1 > 0, \ldots, S_N > 0 } \right] = \mathbf{P} \left[ S_1 > 0, \ldots, S_N > 0 \right] \sim \frac{1}{ \sqrt{2 \pi N}},
\end{equation}
  where the last equivalence is well known (cf \cite{Fe}), which entails that for every $\beta \in \mathbf{R}$ , $F(\beta) \geq 0$. \\ \indent
    In this model, it is possible to show that $\beta_c$ is actually equal to $\log(2)$ (see \cite{GG} chapter 1, or \cite{IY}). 
     We are interested in the scaling limit of the system near criticality. To be more explicit, we define the application: $X^N: \mathbf{R}^N \rightarrow
C([0,1])$ by:
\begin{equation} \label{ent}
  (X_t^N(x))_{t \in [0,1]} = \frac{x_{[Nt]}}{N^{1/2}} + (Nt-[Nt])\frac{
    x_{[Nt]+1}-x_{[Nt]}}{N^{1/2}}.
\end{equation}
 and we introduce the sequence of measures $
 Q_{\beta,N}
:= \mathbf{P}_{\beta,N} \circ (X^N)^{-1}, N \geq 0.$ Of course, in \eqref{ent}, $[x]$ denotes the integer part of the real number $x$.  \\ \indent
  Let $(B_t)_{t \geq 0}$ be a standard Brownian motion defined on
  our probability space $(\Omega,\mathcal{F},\mathcal{P})$ in the
  sequel, and $(L_t)_{t \geq 0}$ denote its local time at zero. 
  It  has been proven in \cite{IY} that the sequence $(Q_{\beta_c,N})_{N \geq 0}$ converges weakly to the law of the process $(|B_t|)_{t \geq 0}$.  This result is actually very natural since it is quite easy to see that  the law of the process $(S_n)_{n \geq 0}$ under $\mathbf{P}_{\beta_c,N}$ has  the same law as $(|S_n|)_{n \geq 0}$ under $\mathbf{P}$.   \\ \indent
 The following result gives some intuition about the more general theorem we are going to prove in the next section in terms of zero level sets. 
 \begin{theorem} Let $\varepsilon \in \mathbf{R}.$ The sequence of measures $ (Q_{\beta_c + \varepsilon N^{-1/2},N})_{N \geq 0}$ converges weakly as $ N \to \infty$; the law of the limiting process is absolutely continuous with respect to $(|B_t|)_{t \geq 0}$ with Radon-Nykodym density given by  $ e^{\varepsilon L_1} / \mathbf{E}\left[ e^{\varepsilon L_1} \right] $. \end{theorem} 
 
 	\begin{figure}[h]
     	\begin{center}
     \input{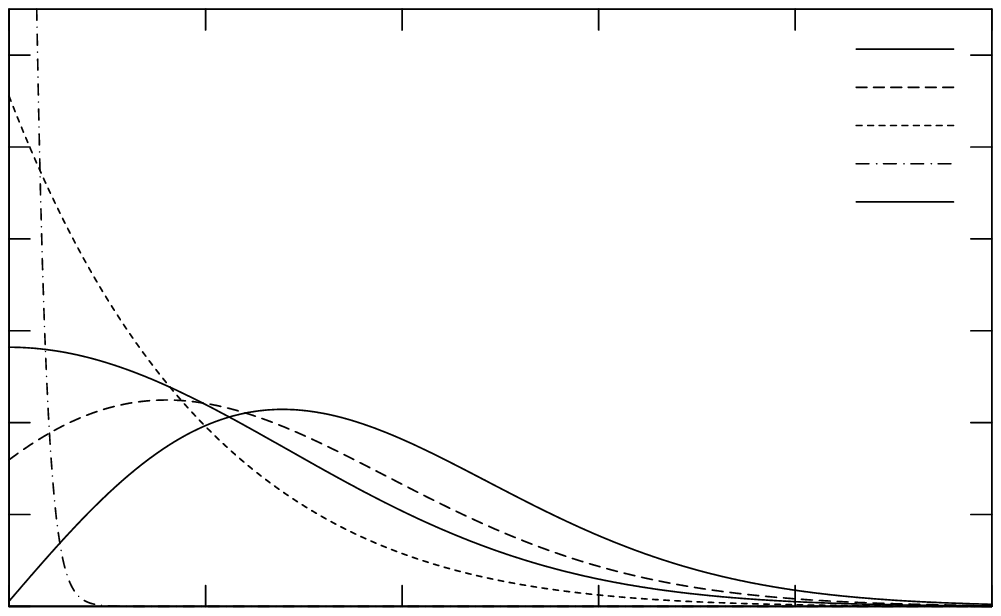}
     	\end{center} 
 \caption{ The distribution of the last point of the system for $ N = 40000$, where $Z_N(x) := \mathbf{E} \left[ \exp( \beta \mathcal{N}_N(\tau) )  \mathbf{1}_{S_N = x} \right]$ for $\beta \in \{ .1, -.2, \varepsilon/\sqrt{40000} \}$ and $\varepsilon \in \{ -1;0;1\}$. }
     	\end{figure}
  
   \begin{proof} 
    The following convergence in law is classical:   
  \begin{equation} \label{CVE}
   \frac{  \mathcal{N}_N }{ \sqrt{N}} \Rightarrow L_1,
 \end{equation} 
  but   a much stronger statement can be proved: given a Brownian motion on $(\Omega,\mathcal{F},\mathcal{P})$, one can construct a simple random walk $(S_n)_{n \geq 0}$ on the same probability space such that, for every $\eta > 0$, one has:
     \begin{equation}
 \lim_{N \to \infty}  N^{-1/4 - \eta} | L_N - \mathcal{N}_N| = 0, 
   \end{equation}
  almost surely as $N \to \infty$ (see \cite{Re}, Theorem 10.1). In particular for $\eta = 1/2$, and by the scaling property of $(L_t)_{t \geq 0}$, this implies immediately the convergence of the joint law $(S_{[Nt]}/\sqrt{N},
 \mathcal{N}_N/\sqrt{N}) \Rightarrow (B_t,L_1)$.\\ \indent
     A proof of the tightness of the sequence $ (Q_{\beta_c + \frac{\varepsilon}{\sqrt{N}},N})_{N \geq 0}$ can be found  in \cite{CaG}. We are therefore left with showing the convergence in law of the finite dimensional marginals of the process, that is we have to show that, for every $n \in \mathbf{N}$ and for every continuous bounded function $F(\cdot)$ from $[0,1]^n$ to $\mathbf{R}$, the following holds:   
  \begin{equation} \label{Cv}
   Q_{\beta_c +\frac{\varepsilon}{\sqrt{N}},N} \left[
   F \left( \omega_{t_1},\ldots,\omega_{t_n} \right) \right] \underset{N \rightarrow
     \infty}{\rightarrow}  \frac{\mathbf{E} \left[ e^{\varepsilon L_1}
   F \left(|B_{t_1}|,\ldots,|B_{t_n}| \right) \right]}{\mathbf{E}[ e^{ \varepsilon L_1}]}.
 \end{equation}  \indent 
 We already know by \cite{IY} that \eqref{Cv} holds for $\varepsilon = 0$.
 The term in the left hand side of \eqref{Cv} can be written as 
 \begin{equation}
   \mathbf{E}_{\beta_c,N} \left[ F \left( \frac{S_{[Nt_1]}}{\sqrt{N}}, \ldots, \frac{S_{[Nt_n]}}{\sqrt{N}} \right)  e^{ \varepsilon \frac{ \mathcal{N}_N}{\sqrt{N}}}  \right] / \mathbf{E}_{\beta_c,N} \left[ e^{ \varepsilon \frac{ \mathcal{N}_N}{\sqrt{N}}}  \right].
\end{equation}
  \indent Since by \ref{CVE} we have $ e^{ \frac{ \varepsilon \mathcal{N}_N}{\sqrt{N}}} \Rightarrow e^{ \varepsilon L_1}$, to prove \ref{Cv}, we just have to show that the family $ \left(e^{ \frac{ \varepsilon \mathcal{N}_N}{\sqrt{N}}} \right)_{ N \geq 0}$ is uniformly integrable. This is obvious if $\varepsilon \leq0$, and will be proved in the last section in a more general setup (see Lemma 2) for $\varepsilon > 0$, although it could be also proven directly using the exact distribution of $\mathcal{N}_N$, which is actually well known in the simple symmetric random walk setup (see for example \cite{Re}). This completes the proof.
       \\ \indent \end{proof} 
 
  \section{The renewal setup}

   \indent 3.1 \textbf{The model.} As we did in the introduction, we denote by $\tau$ the points of a renewal process, that is the
 image set of a 
 sequence of iid positive random variables $(l_i)_{i \geq 0}$ whose
 law is $\mathbf{P}$. More precisely, $\tau := \{ \tau_0,\tau_1,
 \ldots, \}$ where $\tau_k = \sum_{i=1}^k l_i$.  
 In what follows,  $
 \mathbf{P}$ is an integer valued law verifying
  \begin{equation} K(n) := \mathbf{P} (l_1 = n) =
 \frac{L(n)}{n^{1+\alpha}}, \label{1} \end{equation}
  where $L(.)$ is a
 slowly varying function. Furthermore, we allow transience for our renewal process, that is we may have $\sum_{n \geq 1} K(n)  < 1$ (or equivalently $\mathbf{P}(l_1 = \infty) > 0)$. For $N \in \mathbf{N}$, we introduce  $\overline{K}(N) := \sum_{n > N} K(n)$. We recall that a positive measurable function $L(.)$ is called slowly varying if for all $ c > 0$, we have $\lim_{x \rightarrow \infty} L(cx)/L(x) = 1$. Such functions are well known, see \cite{BGT}. One basic example  of slowly varying function (apart from the trivial case of constants) at infinity is $x \mapsto \log(1+x), x \in (0,\infty)$; but logarithmic functions are not the only examples, $x \mapsto \exp(a(\log(1+x))^c)$ is also slowly varying for every $ a \in \mathbf{R}$ and $c < 1$. \\  \indent 
 We then define the homogeneous pinning polymer law exactly as in \eqref{Def}.
  In all what follows, we will denote by $\tau_{(N)}$ the set $(\tau \cap[0,N])/N \subset [0,1]$. When we use this notation, it will be implicit in the sequel that the set $\tau$ has the law $\mathbf{P}_{N,\beta}$, where $\beta$ will be obvious from the context.\\ \indent
   This model has been essentially completely solved, see \cite{GG}; basic facts about it are the existence of a localization delocalization transition for a critical $\beta$, whose value is given by $\beta_c = -\log(\Sigma_K)$, where $\Sigma_K := \sum_{n \geq 1} K(n)$. In particular, we have the equivalence $\beta_c = 0$ iff the underlying renewal process is recurrent.
 This transition is actually defined in terms of the free energy, which is defined exactly as in the previous sections. 
  It is actually possible to give a very intuitive description of both phases in terms of the scaling limit of the zero level set. Intuitively, if the system is localized, we expect the set $\tau_{(N)}$ to converge in a suitable way to the full interval $[0,1]$. Similarly, if $\beta < \beta_c$, we expect it to converge to $\{0\}$. \\ 
 \indent To make such statements quantitative, we introduce $\mathcal{F}$ the set of closed subsets of $[0,1]$. We endow it with the topology of Matheron, which can be described as follows. For $F \in \mathcal{F}$ and $ t \in \mathbf{R}^+$, we set $d_t(F) := \inf( F \cap (t,\infty))$. Notice that $ t \mapsto d_t(F)$ is a right-continuous function and that $F$ can actually be identified with $d_{(\cdot)} (F)$ because $F = \{ t \in \mathbf{R}^+ : d_{t^-}(F) = t \}$. Then in terms of $d_{(\cdot)} (F)$, the Matheron topology is the standart Skorohod topology on c\`{a}dl\`{a}g functions taking values in $\mathbf{R}^+ \cup \{ + \infty\}$. We point out that with this topology, the space $\mathcal{F}$ is metrizable, separable and compact, hence in particular Polish. We denote by $\rho(\cdot)$ the Hausdorff metric given by  $\rho (F_1,F_2) :=
 \max_{i \in \{(1,2),(2,1)\}} \sup_{t \in F_{i_1}} \inf_{s \in F_{i_2}}
 |t-s|$ where  $F_1,F_2 \in \mathcal{F}$. It is then a well known result that this metric is equivalent to the metric engendered by the Skorohod distance via the identification through $d_{(\cdot)}$ on $\mathcal{F}$. Using this topology, both of the above statements have actually been proven in \cite{GG}. 
 The scaling limit at criticality (that is when $\beta = \beta_c$) is much richer than that. To describe it, we need the notion of \textit{$\alpha$ stable sets}. \\ 
   \indent 3.2 \textbf{ $\alpha$ regenerative sets and subordinators.}
Recall that a subordinator is a non decreasing L\'{e}vy process. A well-known fact about a subordinator $(\sigma_t)_{t \geq 0}$  is the existence of its so-called L\'{e}vy Khintchine exponent $\Phi(.)$, that is a measurable function verifying $\forall \lambda > 0, \forall t \geq 0$, $\mathbf{E} \left[ e^{-\lambda \sigma_t} \right] = e^{- t \Phi(\lambda) }$. When $\Phi(t) = t^{\alpha}$ for $\alpha \in (0,1)$, the associated subordinator is refered to as $\alpha$ stable, and we will denote it by $(\sigma_s^{(\alpha)})_{s \geq 0}$ in what follows. We already introduced its \textit{local time} in \eqref{DefTpsLoc}.  These objects are very well-known and a classic reference is \cite{Be}.
    We then define the set $\mathcal{A}_{\alpha} = \overline{ \{ \sigma_s^{(\alpha)}, s \geq 0 \} }$ ($\in \mathcal{F}$) the \textit{regenerative} set of index $\alpha$, where for a subset $ A \subset \mathbf{R}, \overline{A}$ is the closure of $A$. An important property that will turn out to be useful is the fact that $\mathbf{P} \left[ 1 \in \mathcal{A}_{\alpha} \right] =0$ (see \cite{Ke}).\\
  \indent 3.3 \textbf{Scaling limits at and near criticality.} In \cite{GG}, it was shown that, if $\beta = \beta_c$, we have the convergence: 
\begin{equation}
  \tau_{(N)}  \Rightarrow \mathcal{A}_{\alpha}
 \end{equation}
  in the recurrent case (that is $K(\infty) = 0$), and in the transient case 
 \begin{equation}
  \tau_{(N)}  \Rightarrow \tilde{\mathcal{A}}_{\alpha},
 \end{equation} where $\tilde{\mathcal{A}}_{\alpha}$ is a random
 subset of $[0,1]$ whose law is absolutely continuous with respect to
 the law of $\mathcal{A}_{\alpha}$ with Radon-Nykodym density equal to
 $ \left(\alpha \pi / \sin(\alpha \pi) \right) (1-\max(\mathcal{A}_{\alpha}
 \cap [0,1]))^{\alpha}$.\\ \indent 
   A first step towards this convergence has been made in \cite{IY} in a discrete random walk set-up; they actually proved the convergence of the entire process at criticality towards the brownian motion. This work has been strongly generalized in \cite{DGZ}, actually being the most important part of the proof of the convergence of more general pinning models towards the reflected brownian motion at criticality. This was in turn generalized in \cite{CGZ} using powerful renewal techniques, on which the present work is based.  \\
  \indent Let us focus for the moment on the recurrent case, that is we assume for the moment $K(\infty)=0$.  
   First of all, we recall the following result (see \cite{Fe},XIII.6 Theorem 2): it is
   possible to choose a sequence $(a_n)_{n \geq 0}$ such that 
  
   \begin{equation}
  \frac{n \Gamma(1-\alpha)L(a_n)}{ \alpha a_n^{\alpha}} \rightarrow 1, \label{2} 
    \end{equation}  
  and furthermore, as soon as a sequence $(a_n)_{n \geq 0}$ verifies \eqref{2}, we have the convergence :
  \begin{equation}
  \frac{\sum_{i=1}^n l_i}{a_n} \Rightarrow \sigma_1^{(\alpha)} \label{3}   
  \end{equation}  
 where the $(l_i)_{i \geq 0}$ are iid with common law $K(.)$ satisfying \eqref{1}. We then define the sequence $(b_n)_{n \geq 0}$ to be the
 inverse sequence of the $(a_n)_{n \geq 0}$, that is the unique sequence, up to
 asymptotic equivalence, which verifies
 \begin{equation} \label{10}
    a_{[b_n]} \sim b_{[a_n]} \sim n 
 \end{equation}
 as $n \rightarrow \infty$. Its existence is ensured by Theorem
 1.5.12, page 28 of \cite{BGT}.
It is rather easy to see
 that $b_n \sim  (\Gamma(1-\alpha) \overline{K}(n))^{-1}$
 as $n \rightarrow \infty$. Let us quickly prove this fact. Using the definition of
 $(a_n)_{n \geq 0}$, we get
 \begin{equation}
 \frac{b_n \Gamma(1-\alpha) L(a_{[b_n]})}{\alpha a_{[b_n]}^{\alpha}} \sim
 \frac{b_n  \Gamma(1-\alpha) L(n)}{\alpha n^{\alpha}} \sim 1,
 \end{equation}
   and then we note that $ \overline{K}(N) \sim L(N)(\alpha N^{\alpha})^{-1}$ (see \cite{GG},
 P.201).     
 We also used the fact that $x_n \sim y_n$ yields $L(x_n)
 \sim L(y_n)$ as $ n \rightarrow \infty$ for $(x_n)_{n \geq 0}$ and $(y_n)_{n \geq 0}$ positive
 sequences. For $n$ large enough, $x_n/y_n \in [1-\eta;
 1+\eta]$ for a given $\eta > 0$, and the fact that the convergence
 $L(cx)/L(x) \rightarrow 1 $ is uniform in $c$ on every compact
   subset of $\mathbf{R}^+$ (a basic result on slowly varying functions, see \cite{BGT}) yields the assumption.\\ \indent 
 
  To deal with the transient case, we introduce the recurrent law $\mathbf{\tilde{P}}$ given by 
  \begin{equation}
  \mathbf{\tilde{P}}(l_1 = n) := \tilde{K}(n) = \frac{L(n) e^{\beta_c}}{n^{1 + \alpha}} = \frac{L(n) }{n^{1 + \alpha} \Sigma_K}.
\end{equation}  \indent
  We have then exactly the same statements as in Eq.\eqref{2} and \eqref{3} making the obvious changes, and we introduce the sequence $(\tilde{b}_n)_{n \geq 0}$, the analogous of $(b_n)_{n \geq 0}$ for the new sequence  $(\tilde{a}_n)_{n \geq 0}$. We are now ready to state our main result: 
    \begin{theorem}
      Let $\varepsilon \in \mathbf{R}$. The following statements hold:
      
    \begin{enumerate}
    \item if $ K(\infty) = 0$, if $\tau$ is distributed according to $\mathbf{P}_{\beta_c+\varepsilon/b_N,N}$, then 
      \begin{equation}
       \tau_{(N)}
      \Rightarrow \mathcal{B}_{\alpha,\varepsilon},  
      \end{equation}
 where $\mathcal{B}_{\alpha,\varepsilon} (\subset [0,1])$ is the random
 set whose law is absolutely continuous with respect to the law of
 $\mathcal{A}_{\alpha}$ whith Radon Nykodym density equal to
 $\exp(\varepsilon L^{(\alpha)}_1)/\mathbf{E}[\exp(\varepsilon L^{(\alpha)}_1)]$. 
    \item  if $K(\infty) > 0$, if $\tau$ is distributed according to $\mathbf{P}_{\beta_c+\varepsilon/\tilde{b}_N,N}$, then
   \begin{equation}
       \tau_{(N)}
      \Rightarrow \tilde{\mathcal{B}}_{\alpha,\varepsilon},  
      \end{equation} 
     where  $\tilde{\mathcal{B}}_{\alpha,\varepsilon} (\subset [0,1])$  has a Radon-Nykodym density given
 by \newline 
 $\left(\exp(\varepsilon L^{(\alpha)}_1)    (1-\max(\mathcal{A}_{\alpha}
 \cap [0,1]))^{\alpha} \right) / \left(\mathbf{E}[\exp(\varepsilon L^{(\alpha)}_1)  (1-\max(\mathcal{A}_{\alpha}
 \cap [0,1]))^{\alpha}] \right)$. 
    \end{enumerate}
 \end{theorem}

 \section{Proofs}
 
  \begin{lemma} If $K(\infty) = 0$, the following convergence holds:
   \begin{equation} 
     \frac{ |\tau_{(N)}|}{b_N} \Rightarrow L_1^{(\alpha)}.
   \end{equation}  \indent
   \end{lemma} 
   \begin{proof} For $M > 0 $, we have the equality:
    \begin{equation} \label{LC}
      \mathbf{P} \left( \frac{ |\tau_{(N)}|}{b_N} \leq M \right) =
           \mathbf{P} \left(\frac{1}{a_{[Mb_N]}} \sum_{i=1}^{[Mb_N]}
             l_i \geq \frac{N}{a_{[Mb_N]}} \right).
   \end{equation}  The right hand side of \eqref{LC} is easily seen to converge towards $\mathbf{P} \left( M^{\frac{1}{\alpha}} \sigma_1^{(\alpha)} \geq 1
       \right)$. Actually, we have  $
   N/a_{[Mb_N]} \rightarrow M^{-\frac{1}{\alpha}}$ as $ N
     \rightarrow \infty$ (since $Mb_N \sim b_{[M^{\frac{1}{\alpha}}N]}$), and it is easy to see that, if $ X_n \Rightarrow X$, if a 
     deterministic sequence $u_n \rightarrow u$ and the distribution of $X$ is continuous, then $
     \mathbf{P} (X_n \geq u_n) \rightarrow \mathbf{P} (X \geq
     u)$. So we just have to recall that the distribution of $\sigma_1^{(\alpha)}$ is continuous, which is actually a well known fact about stable subordinators (see \cite{Le}, P.271). \\ 
         \indent Because of the scaling property of the stable subordinator, we get 
   \begin{equation*}
     \mathbf{P} \left( M^{\frac{1}{\alpha}} \sigma_1^{(\alpha)} \geq 1
       \right) = \mathbf{P}( \sigma_M^{(\alpha)} \geq 1) 
   \end{equation*}
  
  \begin{equation}
     = \mathbf{P}(L_1^{(\alpha)} \leq M),
   \end{equation} which proves our claim. \end{proof}

  \begin{lemma} For $K(\infty)=0$ and for all $\varepsilon > 0 $, the family
  $ \big(\exp(\frac{ \varepsilon |\tau_{(N)}|}{b_N}) \big)_{N \geq 0}$ is uniformly
  integrable.
    \end{lemma}
  
 \begin{proof}  It is enough to show that, for all $a >0$, we have 

 \begin{equation} \label{UI1}
 \sup_{N \in \mathbf{N}} \mathbf{E} \left[ e^{a |\tau_{(N)}|/b_N} \right]   < \infty.
\end{equation}
     As 
     \begin{equation}   \label{UI2}
      \mathbf{E} \left[ e^{a|\tau_{(N)}|/b_N} \right] = \int_1^{\infty} \mathbf{P} \left[ |\tau_{(N)}|/b_N > 1/a \log(x) \right] \textrm{d}x,
      \end{equation}
       to show \ref{UI1}, we just have to see that for all $C_1 >0$, there exists $C_2 >0$ such that  for $y \geq 1$,  
       \begin{equation} \label{AV}
     \mathbf{P} \left[ |\tau_{(N)}|/b_N > y \right] \leq C_2 \exp( -C_1 y),
      \end{equation}
        as soon as $N$ is large enough.    
     Using the Markov inequality, we see that for  every $
    \lambda >0$ and $y \geq 1$ :
    \begin{equation} \label{MI}
         \mathbf{P} \left( \frac{|\tau_{(N)})|}{b_N} > y \right) =
         \mathbf{P} \left(
  \sum_{i=1}^{[yb_N]} l_i < N \right) \leq  \mathbf{E} \left[ e^{-\lambda l_1} \right]^{[yb_N]} e^{\lambda N}. 
  \end{equation}
       Noting that $\tau$ is recurrent, we have :
        \begin{equation}
        \log(\mathbf{E}[e^{-\lambda l_1}]) = \log  \left( 1 - L(1/\lambda) \lambda^{1 + \alpha} \sum_{n \geq 1} \frac{1-e^{-\lambda n } }{(\lambda n)^{1+\alpha}} \frac{ L(\frac{n\lambda}{\lambda})} {L(1/\lambda)} \right).
         \end{equation}
      For any $\eta > 0$ and using the uniform convergence property on compact sets for slowly varying functions, one has the equivalence (for $\lambda \searrow 0$)
         \begin{equation}
         \lambda \sum_{n\lambda  \geq \eta}^{1/\eta} {1-e^{-\lambda n } \over (\lambda n)^{1+\alpha}} \frac{ L({n\lambda \over \lambda})} {L(1/\lambda)} \sim \int_{\eta}^{1/\eta} \frac{1-e^{x}}{x^{1+ \alpha}} dx,
          \end{equation}
        and it is easily seen that this implies the equivalence (still for  $\lambda \searrow 0$):
         \begin{equation}  \label{EQV} -\log(\mathbf{E}[e^{-\lambda l_1}])
    \sim  \Gamma(1-\alpha)/\alpha \lambda^{\alpha}
    L(1/\lambda).
     \end{equation}

        This entails that there exists $ c(\alpha) > 0$ and $\lambda_0(\alpha) > 0$ such that for all  $  \lambda \in (0,\lambda_0)$, we have $ -\log(\mathbf{E}[e^{-\lambda l_1}])
    \geq c(\alpha) \lambda^{\alpha}
    L(1/\lambda)$. Note that  $ (b_N L(N)N^{-\alpha})_{N \geq 0}$ is a positive bounded sequence. Then, using \ref{MI}, we
    get, for every $N$:      
    
 \begin{equation} \label{UI3}
 \mathbf{P} \left(  \frac{|\tau_{(N)})|}{b_N} > y \right) \leq
   \inf_{0 \leq \lambda \leq \lambda_0} \exp \left( - c(\alpha) \lambda^{\alpha}L(1/\lambda) y \frac{N^{\alpha}}{L(N)} + \lambda N \right),
\end{equation}
  possibly modifying $c(\alpha)$ in doing so. Then we consider $N_0$ such that for $N \geq N_0$, $C_1/N < \lambda_0$, and we give us $C_2 \in (C_1; N_0  \lambda_0)$. For $u$ in the compact set $\left[ C_1,C_2 \right]$, the convergence of $L(N/u)/L(N)$ towards $1$ is uniform, so that, for $\eta > 0$, for $N$ large enough, $L(N/\lambda N)/L(N) \geq 1 - \eta $ as long as $ C_1/N \leq \lambda \leq C_2/N$. Thus, for $N$ large:
  \begin{equation}
    \mathbf{P} \left(  \frac{|\tau_{(N)})|}{b_N} > y \right) \leq \inf_{C_1/N \leq \lambda \leq C_2/N} \exp \left( -c(\alpha)(1-\eta) (\lambda N)^{\alpha} y + \lambda N \right).
  \end{equation}
   And this last inequality clearly entails \ref{AV}, and thus the lemma.
    \end{proof} 
 
 \begin{lemma}
 For $K(\infty) = 0$, we have the convergence: 
  \begin{equation} \label{AM}
 ( |\tau_{(N)}|/b_N,\tau_{(N)}) \Rightarrow 
   (L_1^{(\alpha)},\mathcal{A}_{\alpha}).
\end{equation}
 \end{lemma}
 
  \begin{proof}
   We introduce the process $(N_{\gamma_N}(t))_{t \geq 0}$,
a Poisson process of rate $ \gamma_N $ where we define  $
 \gamma_N := \left( \sum_{n \geq 1} (1-e^{-n/N})K(n) \right)^{-1}$. We take $(N_{\gamma_N}(t))_{t \geq 0}$ independent from $\tau$. It is easy to see that the process $\left(\tau_{N_{\gamma_N}(t)}/N \right)_{t \geq 0}$ is a subordinator. In \cite{GG}, it was shown that the L\'{e}vy exponent of 
$ \left(\tau_{N_{\gamma_N}(t)}/N \right)_{t \geq 0}$ converges to the one of $ \left(\sigma_s^{(\alpha)} \right)_{s \geq 0}$, which implies $ \overline{ \{ \tau_{N_{\gamma_N}(s)}/N ,  s \geq 0 \} } \Rightarrow \mathcal{A}_{\alpha}$  as well as the convergence in law of the entire process $ \left(\tau_{N_{\gamma_N}(t)}/N \right)_{t \geq 0}$ towards $ \left(\sigma_s^{(\alpha)} \right)_{s \geq 0}$ (see \cite{Fi}) in the Polish space $\mathcal{D}$. Here, $\mathcal{D}$ denotes the space
  of c\`{a}dl\`{a}g functions on $[0,1]$ endowed with the standart Skorokhod topology (that is $d(f,g) := \inf_{ \lambda \in \Lambda} \{ \max(||
  \lambda - id||,|| f \circ \lambda - g||) \}$ where $ \Lambda $
  is the set of non-decreasing homeomorphisms from $[0,1]$ to $[0,1]$) and of course $|| f || = \sup_{t \in [0,1]} |f(t)|$ for $f \in \mathcal{D}$.  \\
 \indent We define the function $(F,G) : \mathcal{D} \rightarrow 
    \mathbf{R} \times \mathcal{F} $ by 
     \begin{equation}
   (F,G)(f):     \indent f  \mapsto  ( 
    \sup \left\{s \geq 0 | f(s) \leq 1 \right\} , \text{Im}(f) ),
\end{equation}
   where Im($f$) is the image set of $f$ where 
   we endow the space $\mathbf{R} \times \mathcal{F}$ by the topology   $\sigma \left(\mathcal{B}(\mathbf{R}) \otimes \mathcal{F} \right)$, and using the distance
   $\max(|.|,\rho)$ where $\rho$ is the Hausdorff metric. We will show that, for all continuous bounded $\mathcal{G} : \mathbf{R} \times \mathcal{F} \rightarrow \mathbf{R}$, we have:  
 \begin{equation} \label{5}
 \mathbf{E} \left[\mathcal{G} \left(\frac{|\tau_{(N)}|}{b_N},\tau_{(N)} \right) \right] =
  \mathbf{E} \left[ \mathcal{G} \left( \Big( F,G \Big) \left((\frac{\tau_{N_{\gamma_N}(s)}}{N})_{s \geq 0} \right) \right) \right] + (o_N(1)),
\end{equation} that is that we can consider the joint law in the left hand side of \eqref{AM} as a function of $(\tau_{N_{\gamma_N}(s)}/N)_{s \geq 0}$ up to negligible corrections.\\ \indent Assume for the moment that \eqref{5} is true. Taking into account the convergence $ \left(\tau_{N_{\gamma_N}(t)}/N \right)_{t \geq 0} \Rightarrow  \left(\sigma_s^{(\alpha)} \right)_{s \geq 0}$ and the fact that $\left(F,G \right) \left( \left(\sigma_s^{(\alpha)} \right)_{s \geq 0} \right)$ is exactly the limiting law we want, it is enough to show that $\left(F,G \right)( \cdot )$ is almost surely continuous at $\left(\sigma_s^{(\alpha)} \right)_{s \geq 0}$ to apply the continuous mapping theorem (see \cite{Bi}) to prove our lemma. So let us show this almost sure continuity first, then we will turn to the proof of \eqref{5}. \\ 
 \indent For $\eta  \in (0,1)$, we denote by $\mathcal{B}_{\eta}$ the subset of
  $\mathcal{D}$ defined by 
  \begin{equation}
    \mathcal{B}_{\eta} := \{ f \in \mathcal{D} | [1-\eta;1+\eta] \cap
    \text{Im(f)} = \emptyset  \}.
  \end{equation}
  We already pointed out that $\mathbf{P}( 1 \in \mathcal{A}_{\alpha})
   = 0$, so that with probability one,
   $(\sigma_s^{(\alpha)})_{s \geq 0} \in \cup_{\eta > 0} \mathcal{B}_{\eta}$. 
   We show that $(F,G)$ restricted to
   $\cup_{\eta > 0} \mathcal{B}_{\eta}$ is continuous.\\ \indent
  Let $f \in \mathcal{B}_{\eta}$ with $\eta > 0$, let us show that $F$
  is continuous at $f$. We introduce $s_f :=\sup \{ s \geq 0 | f(s)
  \leq 1 \}$. Let $g \in \mathcal{D}$ such that $d(f,g) <
  \mu/2$ where $ \mu \in (0,\eta)$, and in the same way we define $s_g :=\sup \{ s \geq 0 | g(s)
  \leq 1 \}$. We consider $\lambda \in
  \Lambda$ such that $||f - g \circ \lambda|| < \mu$ and $|| \lambda
  - id|| < \mu$. Furthermore, for all $ \xi > 0$, and as $f \in
  \mathcal{B}_{\eta}$, we have $f(s_f-\xi) < 1-\eta$ and $f(s_f) > 1 +
  \eta$. Then we get
  \begin{equation}
     g(\lambda(s_f-\xi)) \leq f(s_f-\xi) + \mu \leq
  1-\eta + \mu  < 1,
  \end{equation}
   and similarly 
   \begin{equation}
    g(\lambda(s_f)) \geq f(s_f) - \mu > 1 + \eta - \mu > 1, 
   \end{equation}
  which gives
  \begin{equation}
    \lambda(s_f-\xi) \leq s_g \leq \lambda(s_f),
  \end{equation}
 and this entails the desired continuity since $\xi$ is arbitrarily
 small, $\lambda$ is continuous and  $|| \lambda - id|| < \mu$.\\
  Note that, by the virtue of the continuous mapping's theorem, this implies the convergence 
 \begin{equation} \label{cvLT}
 \sup \left\{s \geq 0,\frac{\tau_{N_{\gamma_N}(s)}}{N}\leq 1 \right\}  \Rightarrow L_1^{(\alpha)}.
\end{equation} \indent
 For the second component, it is very easy also; for $g$ as above, we have:
 \begin{equation}
   \rho \left( \text{Im(f)}, \text{Im(g)} \right) = \max \left( \sup_{s \in \mathbf{R}} \inf_{t \in
     \mathbf{R}} |f(t)-g(s)|, \sup_{s \in \mathbf{R}} \inf_{t \in
     \mathbf{R}} |f(s)-g(t)| \right).
 \end{equation}
   And for $\lambda$ as above and $s \in \mathbf{R}$, we have $\inf_{t
     \in \mathbf{R}} |f(s)-g(t)| \leq |f(s)-g(\lambda(s))| \leq
   \mu$, which achieves the proof of the continuity of the
   couple $(F,G)$ on $\cup_{\eta > 0} \mathcal{B}_{\eta}$. \\ \indent We now go to the proof of \eqref{5}.
 Let $\delta >0$, there exists a compact set $ K_{\delta} \subset \mathbf{R}^+$ such that for all  $ N \in \mathbf{N}$, we have:
 \begin{equation} \label{6}
 \max \left( \mathbf{P} \left[ \frac{|\tau_{(N)}|}{b_N} \in K_{\delta}^c \right] , \mathbf{P} \left[ \sup \left\{s \geq 0 \Big| \frac{\tau_{N_{\gamma_N}(s)}}{N}\leq 1 \right\} \in K_{\delta}^c \right] \right) < \delta.  
\end{equation}  
 This is due to the fact that the sequences $ \left(\sup \left\{s \geq 0 \Big| \tau_{N_{\gamma_N}(s)}/N\leq 1 \right\} \right)_{N \geq 0}$ and $ \left( \sup \left\{ \frac{n}{b_N} \Big| \frac{|  
   \tau_n|}{N} \leq 1 \right\} \right)_{N \geq 0} $ are tight because they converge in law (see \eqref{cvLT} and Lemma 1). Similarly, the convergence of the sequence $(\tau_{(N)})_{N \geq 0}$ implies the existence of a compact set $L_{\delta} \subset \mathcal{F}$ such that for every $ N \in \mathbf{N}$,
   \begin{equation} \label{7}
   \mathbf{P} \left[ \tau_{(N)}  \in L_{\delta}^c \right] < \delta.
\end{equation}
 For $N \in \mathbf{N}$, we introduce the event \begin{equation}
 \mathcal{H}_{N,\delta} := \left\{ \frac{|\tau_{(N)}|}{b_N} \in K_{\delta}, \sup \left\{s \geq 0 \Big| \frac{\tau_{N_{\gamma_N}(s)}}{N}\leq 1 \right\} \in K_{\delta}, \tau_{(N)} \in L_{\delta} \right\}.
\end{equation}
 Eqs. \eqref{6} and \eqref{7} say that $\mathbf{P} \left[ \mathcal{H}_{N,\delta}^c \right] < \delta.$\\ \indent
 Let $M > 0$ be an upper bound for $\big| \mathcal{G} \big|$.
  For every $ N \in \mathbf{N}$ and taking into account the fact that $\gamma_N \sim b_N$ (see \cite{GG}, (A.48)), we have: 
 \begin{align} \begin{split} \label{8}
  & \quad \Bigg| \mathbf{E} \left[ \mathcal{G} \left(\sup \left\{ s\geq 0 \Bigg| \frac{\tau_{N_{\gamma_N}(s)}}{N} \leq 1 \right\}, \tau_{(N)} \right) -  \mathcal{G} \left( \frac{|\tau_{(N)}|}{\gamma_N}, \tau_{(N)}  \right)  \right] \Bigg| \\
  & \leq  \mathbf{E} \Bigg[ \Bigg|\mathcal{G} \left(\sup \left\{ s\geq 0 \Bigg| \frac{\tau_{N_{\gamma_N}(s)}}{N} \leq 1 \right\}, \tau_{(N)} \right) -  \mathcal{G} \left( \frac{|\tau_{(N)}|}{\gamma_N}, \tau_{(N)} \right) \Bigg| \1_{\mathcal{H}_{N,\delta}} \Bigg] + 2M \delta  \\  
  & \leq  \mathbf{E} \Bigg[ \Bigg|\mathcal{G} \left(\sup \left\{ s\geq 0 \Bigg| \frac{\tau_{N_{\gamma_N}(s)}}{N} \leq 1 \right\}, \tau_{(N)} \right)  \\  
 & -  \mathcal{G} \left( \sup \left\{ \frac{N_{\gamma_N}(s)}{\gamma_N} -s +s \Bigg| \frac{\tau_{N_{\gamma_N}(s)}}{N} \leq 1 \right\}, \tau_{(N)} \right) \Bigg| \1_{\mathcal{H}_{N,\delta}} \Bigg] + 2M \delta. 
  \end{split}
\end{align}
    \indent For a given $\kappa > 0$, we introduce the event $\mathcal{J}_{\kappa,\delta}$ defined by:
    \begin{equation}
  \mathcal{J}_{\kappa,\delta} := \left\{ \sup_{ s \in K_{\delta} }\Bigg|\frac{N_{\gamma_N}(s)}{\gamma_N} - s \Bigg| \geq
    \kappa \right\}.
\end{equation}  \indent
  Note that the process $\left(\frac{N_{\gamma_N}(s)}{\gamma_N} - s\right)_{s \geq 0}$ is a martingale, so that, applying Doob's inequality, we get: 
 \begin{multline}
  \mathbf{P} \left[ \mathcal{J}_{\kappa,\delta} \right] = \mathbf{P} \left[ \sup_{ s \in K_{\delta} }\Bigg|\frac{N_{\gamma_N}(s)}{\gamma_N} - s \Bigg|^2 \geq
    \kappa^2 \right]  \\
  \leq \frac{1}{\kappa^2} \sup_{ s \in K_{\delta} } \mathbf{E} \left[ \Bigg|\frac{N_{\gamma_N}(s)}{\gamma_N} - s \Bigg|^2  \right] = \frac{ \max K_{\delta} }{\gamma_N \kappa^2}. \end{multline}

  Cutting once again the expectation appearing in Eq.\eqref{8} with respect to the fact that $\mathcal{J}_{\kappa,\delta}$ occurs or not, we get: 
 \begin{align} \begin{split}
  & \quad \bigg| \mathbf{E} \left[ \mathcal{G} \left(\sup \left\{ s\geq 0 \Bigg| \frac{\tau_{N_{\gamma_N}(s)}}{N} \leq 1 \right\}, \tau_{(N)} \right) -  \mathcal{G} \left( \frac{|\tau_{(N)}|}{\gamma_N}, \tau_{(N)} \right)  \right] \bigg| \\
  & \leq \mathbf{E} \Bigg[ \sup_{ |u| \leq \kappa} \Bigg\{ \Bigg| \mathcal{G} \left(\sup \left\{ s\geq 0 \Bigg| \frac{\tau_{N_{\gamma_N}(s)}}{N} \leq 1 \right\}, \tau_{(N)} \right)\\  
 & -  \mathcal{G} \left( u +\sup \left\{ s \geq 0 \Bigg| \frac{\tau_{N_{\gamma_N}(s)}}{N} \leq 1 \Bigg\}, \tau_{(N)} \right) \Bigg| \right\} \1_{\mathcal{H}_{N,\delta}} \1_{\mathcal{J}_{\kappa,\delta}^c} \Bigg] \\
 & + 2M \left(\delta + \frac{ \max K_{\delta} }{\gamma_N \kappa^2} \right). \\
 \end{split}
\end{align}

 In the last inequality, we consider first $\kappa$ very small, which deals with the expectation term due to the uniform continuity of $\mathcal{G}$ on the compact set $K_{\delta} \times L_{\delta}$. Then, we consider $\delta$ small enough, and this achieves the proof of \eqref{5}, since $\gamma_N \stackrel{N \to \infty}{\to} \infty$. Thus the lemma is proved. \end{proof} 
\indent  We are now ready to show the main part of our
   theorem.
  
  \begin{proof}[Proof of Theorem 2]

   We have just shown the following convergence: under the
   law $K(.)$, assuming $K(\infty)=0$, for all $  F $ continuous bounded function on $\mathcal{F}$ and for all $
   \varepsilon \in \mathbf{R}$, we have: 
  
 \begin{equation}
     \mathbf{E} \left[ F(\tau_{(N)})e^{ \varepsilon |\tau_{(N)}|/b_N} \right] \underset{N
       \rightarrow  \infty}\rightarrow
     \mathbf{E} \left[ F(\mathcal{A}_{\alpha})e^{\varepsilon L_1^{(\alpha)}} \right].
   \end{equation} 
 Note that the case where $\varepsilon > 0$ actually uses the fact (see Lemma 2) that the sequence $(e^{\varepsilon |\tau_{(N)}|/b_N})_{N \geq 0}$ is uniformly integrable. \\ \indent

 We then have
 \begin{equation}
   \mathbf{E} \left[ F(\tau_{(N)}) \right] =
   \frac{Z_{N,\beta_c}}{Z_{N,\beta_c+\varepsilon/b_N}}
    \times \frac{1}{Z_{N,\beta_c}} \mathbf{E} \left[ F(\tau_{(N)}) e^{\varepsilon |\tau_{(N)}|/b_N} \right] .
 \end{equation}

  As in the case $K(\infty)=0$, $\mathbf{P}_{N,\beta_c}$ is actually
  $\mathbf{P}$ (because $ \beta_c = -\log(\Sigma_K) = 0$) and thanks to the easy remark that
  $Z_{N,\beta_c+ \varepsilon/b_N} = \mathbf{E} \left[ e^{\varepsilon |\tau_{(N)}|/b_N} \right] \underset{N
       \rightarrow  \infty}\rightarrow \mathbf{E} \left[
     e^{\varepsilon L_1^{(\alpha)}} \right]$, we have actually shown the first part of the
       claim. We also 
       used the fact that $Z_{N,\beta_c} = 1$ for all $N \in
       \mathbf{N}$  see (\cite{GG}, (2.17)).

    We now deal finally with the case $K(\infty) > 0$; we can write: 
     \begin{equation} \label{Eq2}
       \mathbf{E}_{N,\beta_c+ \varepsilon/\tilde{b}_N} \left[F(\tau_{(N)}) \right] =
       \mathbf{\tilde{E}} \left[ F(\tau_{(N)})
       e^{\varepsilon |\tau_{(N)}|/\tilde{b}_N}  \frac{
         \overline{K}(N(1- \max(\tau_{(N)}))) +
         K(\infty)}{Z_{N,\beta_c+ \varepsilon/\tilde{b}_N} \sum_{j>N(1-\max(\tau_{(N)}))}\tilde{K}(j)} \right].
     \end{equation}
  The justification for \eqref{Eq2} can be found in \cite{GG}, equality (2.48). 
  Then, on the event $\max \tau_{(N)} < 1 - \eta$ for a given $\eta \in (0,1)$, it is known that ((2.50) of \cite{GG})) 
  \begin{equation}  
         \frac{\overline{K}(N(1- \max(\tau_{(N)}))) +
         K(\infty)}{Z_{N,\beta_c}
         \sum_{j>N(1-\max(\tau_{(N)}))}\tilde{K}(j)}
       \underset{N \rightarrow \infty}\sim   \frac{\alpha
         \pi}{\sin(\alpha \pi)}(1-\max(\tau_{(N)}))^{\alpha},
  \end{equation} uniformly in the trajectories of $\tau$. 
  Using once again the convergence of the joint law $(\tau_{(N)},|\tau_{(N)}|/\tilde{b}_N)$, where this time $\tau$ is distributed according to $\mathbf{\tilde{P}}$, we get :    
 \begin{multline}  \label{Eq3}
       \mathbf{E}_{N,\beta_c+\varepsilon \tilde{b}_N} \left[F(\tau_{(N)}) \1_{ \max \tau_{(N)} < 1 - \eta } \right] \underset{N \rightarrow  \infty}{\longrightarrow}
 \\ 
        \mathbf{E} \left[ F(\mathcal{A}_{\alpha}) e^{ \varepsilon L_1^{(\alpha)}} \1_{ \max \mathcal{A}_{\alpha} < 1 - \eta } \frac{\alpha \pi}{\sin(\alpha \pi)}(1-\max(\mathcal{A}_{\alpha}))^{\alpha} \right].
   \end{multline}
   The fact that $\max(\mathcal{A}_{\alpha}) < 1$ almost surely allows us to take the above limit for $\eta \searrow 0$. 
 Taking $F \equiv 1$ in \eqref{Eq3}, we get 
  \begin{equation}
    \frac{Z_{N,\beta_c+ \varepsilon/\tilde{b}_N}}{Z_{N,\beta_c}}
     \underset{N \rightarrow \infty} \rightarrow \mathbf{E} \left[
     e^{\varepsilon L_1^{(\alpha)}}\frac{\alpha
         \pi}{\sin(\alpha \pi)}(1-\max(\mathcal{A}_{\alpha}))^{\alpha} \right].
  \end{equation}

 This achieves the proof because we can write

 \begin{multline}
   \mathbf{E}_{N,\beta_c+ \varepsilon/\tilde{b}_N} \left[ F(\tau_{(N)}) \right] \\
   = \frac{Z_{N,\beta_c}}{Z_{N,\beta_c+\varepsilon/\tilde{b}_N}}
     \mathbf{\tilde{E}}\left[ F(\tau_{(N)})
       e^{\varepsilon |\tau_{(N)}|/\tilde{b}_N}  \frac{
         \overline{K}(N(1- \max(\tau_{(N)}))) +
         K(\infty)}{Z_{N,\beta_c} \sum_{j>N(1-\max(\tau_{(N)}))}\tilde{K}(j)} \right],
 \end{multline}
  and the term in the right hand side converges towards the desired quantity.

 \end{proof}
 
 \textbf{Acknowledgement} I am very grateful to my Ph.D. supervisor Giambattista Giacomin for his help and suggestions. I would like to thank Francesco Caravenna also, in particular for the proof of Lemma 2. \\
 \indent I am also grateful to an anonymous referee for his remarks and comments.

\bibliographystyle{alea2}
\bibliography{bibliotpsloc}

\end{document}